\documentclass[12pt,a4paper,oneside,reqno]{amsproc}
\usepackage{amsmath}
\usepackage{amssymb,amsfonts,mathrsfs}
\usepackage[colorlinks=true,linkcolor=blue,citecolor=blue]{hyperref}
\usepackage[driver=pdftex,lmargin=2.5cm,rmargin=2.5cm,heightrounded=true,centering]{geometry}
\usepackage{math50}
\usepackage{local}
\usepackage{epsfig}
\usepackage{graphicx}
\usepackage{xy}
\xyoption{all}
\begin{document}
\title[Ground state solutions for non-autonomous fractional Choquard equations]
{Ground state solutions for non-autonomous fractional Choquard equations}
\author{Yan-Hong Chen}
\address{School of Mathematical Science, Nankai University, Tianjin 300071, P.R. China}
\email{cyh1801@163.com}
\author{Chungen Liu}
\address{School of Mathematical Science and LPMC, Nankai University, Tianjin 300071, P.R. China}
\email{liucg@nankai.edu.cn}
\newcommand{\optional}[1]{\relax}
\setcounter{secnumdepth}{3}
\setcounter{section}{0} \setcounter{equation}{0}
\numberwithin{equation}{section}
\newcommand{\MLversion}{1.1}
\thanks{The second author was partially supported by NFSC (11071127, 10621101), the 973 Program of STM of China (2011CB808002) and SRFDP}
\begin{abstract}
We consider the following nonlinear fractional Choquard equation,
\begin{equation}\label{e:introduction}
\begin{cases}
(-\Delta)^{s} u + u = (1 + a(x))(I_\alpha \ast (|u|^{p}))|u|^{p - 2}u\quad\text{ in }\mathbb{R}^N,\\
u(x)\to 0\quad\text{ as }|x|\to \infty,
\end{cases}
\end{equation}
here $s\in (0, 1)$, $\alpha\in (0, N)$, $p\in [2, \infty)$ and $\frac{N - 2s}{N + \alpha} < \frac{1}{p} < \frac{N}{N + \alpha}$. Assume $\lim_{|x|\to\infty}a(x) = 0$
and satisfying suitable assumptions but not requiring any symmetry property on $a(x)$, we prove the existence of ground state solutions for (\ref{e:introduction}).
\end{abstract}
\maketitle
\section{Introduction}
The nonlinear Choquard or Choquard-Pekar equations are of form
\begin{equation}\label{333e:main0}
\begin{cases}
- \Delta u + u = (I_\alpha \ast |u|^{p})|u|^{p - 2}u\quad\text{ in }\mathbb{R}^N,\\
u(x)\to 0\quad\text{ as }|x|\to \infty.
\end{cases}
\end{equation}
Here $\alpha\in (0, N)$, $p\in (1, \infty)$, $I_\alpha:\mathbb{R}^N\to \mathbb{R}$ is the Riesz potential defined by
\begin{equation}\label{333e:rieszpotential}
I_\alpha(x) = \frac{\Gamma(\frac{N - \alpha}{2})}{\Gamma(\frac{\alpha}{2})\pi^{N/2}2^\alpha|x|^{N - \alpha}},
\end{equation}
and $\Gamma$ is the Gamma function, see \cite{MRieszActaMath}. It is well known that if $u$ solves $(\ref{333e:main0})$ when $\alpha = 2$, $N\geq 3$, then $(u, v) = (u, I_\alpha \ast |u|^p)$ satisfies the system
\begin{equation*}
\begin{cases} -\Delta u + u =  v|u|^{p - 2}u\quad\text{in } \mathbb{R}^N,\\
-\Delta v =  |u|^p\quad\text{in } \mathbb{R}^N,\\
u(x)\to 0\quad\text{as } |x|\to \infty,\\
v(x)\to 0\quad\text{as } |x|\to\infty.
\end{cases}
\end{equation*}

(\ref{333e:main0}) has several physical origins. In the case $N = 3$, $p = 2$ and $\alpha = 2$, the problem
\begin{equation}\label{333e:phyiscal}
\begin{cases}
- \Delta u + u = (I_2 \ast |u|^{2})u\quad\text{ in }\mathbb{R}^3,\\
u(x)\to 0\quad\text{ as }|x|\to \infty
\end{cases}
\end{equation}
appeared in \cite{Pekar1954} by Pekar when he described the quantum mechanics of a polaron. In the approximation to Hartree-Fock theory of one component plasma, Choquard used (\ref{333e:phyiscal}) to describe an electron trapped in its own hole, see  \cite{LiebElliot}. In \cite{MorozPenrose}, Penrose proposed (\ref{333e:phyiscal}) as a model of self-gravitating matter in which quantum state reduction was understood as a gravitational phenomenon. Equations of type (\ref{333e:main0}) are usually called the Schr\"{o}dinger-Newton equation. If $u$ solves (\ref{333e:main0}), then the function $\psi$ defined by $\psi(t, x) = e^{it}u(x)$ is a solitary wave solution of the focusing time-dependent Hartree equation
$$
i\psi_t = -\Delta \psi - (I_\alpha\ast |\psi|^p)|\psi|^{p - 2}\psi\text{ in }\mathbb{R}_+ \times \mathbb{R}^N.
$$
So (\ref{333e:main0}) is also known as the stationary nonlinear Hartree equation.

In \cite{LiebElliot}, Lieb proved that the ground state of (\ref{333e:phyiscal}) is radial and unique up to translations; later, in \cite{LionsPLCERQNA1980}, Lions proved the existence of infinitely many radially symmetric solutions to (\ref{333e:phyiscal}); in \cite{WeiWinter2009}, Wei and Winter showed the nondegeneracy of the ground state and studied the multi-bump solutions for (\ref{333e:phyiscal}); in \cite{MaLiCPSSNCEARMA2010}, Ma and Zhao proved, under some assumptions on $N$, $\alpha$ and $p$, that every positive solution of (\ref{333e:main0}) is radially symmetric and monotone decreasing about some point by the method moving planes in an integral form developed in \cite{ChenWenxiongLicongmingOubiaoCPAM2006}; in \cite{CingolaniClappSecchiZAMP}, Cingolani, Clapp and Secchi proved some existence and multiplicity results, and established the regularity and some decay asymptotics at infinity of the ground states for (\ref{333e:main0}) in the electromagnetic case. In \cite{MorozSchaftingen2013}, Moroz and Schaftingen considered problem (\ref{333e:main0}), they eliminated the restriction of \cite{MaLiCPSSNCEARMA2010},  proved  the regularity, positivity and radial symmetry of the ground states for optimal range of parameters, the decay asymptotics at infinity of the ground states were also derived. In \cite{MorozSchaftingen2013JDE}, Moroz and Schaftingen showed that for some values of the parameters, (\ref{333e:main0}) does not have nontrivial nonnegative super solutions in exterior domains. In \cite{MorozSchaftingenpreprint}, Moroz and Schaftingen proved the existence of ground state solutions to the nonlinear Choquard equation
\begin{equation}
\begin{cases}
- \Delta u + u = (I_\alpha \ast F(u))F'(u)\quad\text{ in }\mathbb{R}^N,\\
u(x)\to 0\quad\text{ as }|x|\to \infty
\end{cases}
\end{equation}
under general conditions on the nonlinearity $F(u)$ in the spirit of Berestycki and Lions in \cite{BerestyckiLionsARMA1983}. In \cite{ClappSalazar2013}, Clapp and Salazar considered the following equation in exterior domains,
\begin{equation}\label{333e:main3}
- \Delta u + W(x)u = (I_\alpha \ast |u|^p)|u|^{p - 2}u\quad  u\in H^1_0(\Omega).
\end{equation}
They established the existence of a positive solution and multiple sign changing solutions for (\ref{333e:main3}). Recently, Moroz and Schaftingen studied the following equation
\begin{equation*}\label{33333e:main3}
- \varepsilon^2\Delta u + V(x)u = \varepsilon^{-\alpha}(I_\alpha \ast |u|^p)|u|^{p - 2}u\quad  \text{in  } \mathbb{R}^n
\end{equation*}
and proved the existence of semi-classical solutions, see \cite{MorozVanSchaftingenCVPDE2015}. In \cite{AveniaSicilianoSquassinaMMMAS2015}, d'Aenia, Siciliano and Squassina obtained regularity, existence, nonexistence, symmetry and decay properties of solutions for the fractional Choquard equation
\begin{equation*}\label{33333e:main3}
(- \Delta)^s u + \omega u = (I_\alpha \ast |u|^{p})|u|^{p - 2}u\quad\text{ in }\mathbb{R}^N.
\end{equation*}

In this paper, we study the following non-autonomous nonlinear fractional Choquard equation
\begin{equation}\label{333e:main}
\begin{cases}
(- \Delta)^s u + u = (1 + a(x))(I_\alpha \ast |u|^{p})|u|^{p - 2}u\quad\text{ in }\mathbb{R}^N,\\
u(x)\to 0\quad\text{ as }|x|\to \infty.
\end{cases}
\end{equation}
Here we assume that $s\in (0, 1)$, $\alpha\in (0, N)$, $p\in [2, \infty)$, $\frac{N - 2s}{N + \alpha} < \frac{1}{p} < \frac{N}{N + \alpha}$, $a = a(x)$ is a scalar function and satisfies the following conditions:
\begin{enumerate}
\item[(a1)]$a(x)\in L^\infty(\mathbb{R}^N)$, $\lim_{|x|\to +\infty}a(x) = 0$;
\item[(a2)]$a(x)\geq 0$, $a(x) > 0$ on a positive measure set and $a(x)\in L^{\frac{2N}{N - Np + 2sp + \alpha}}(\mathbb{R}^N)$.
\end{enumerate}
Our main theorem is
\begin{theorem}\label{333t:theorem1}
Assume $s\in (0, 1)$, $\alpha\in (0, N)$, $p\in [2, \infty)$. If $\frac{N - 2s}{N + \alpha} < \frac{1}{p} < \frac{N}{N + \alpha}$, $a(x)$ satisfies conditions (a1) and (a2), then problem (\ref{333e:main}) has at least one positive ground state solution.
\end{theorem}

\begin{remark}
If $s = 1$, Theorem \ref{333t:theorem1} was proved by P. L. Lions, see \cite{LionsPLCERQNA1980}. In \cite{ZhangZhengjieTassiloHuailianXiaHongqiang}, the authors gave an extension of Lions's result by a Min-Max method argument (also in the case $s = 1$).
\end{remark}

(\ref{333e:main}) has a variational structure: critical points of the functional $E_{\alpha, p}\in C^1(H^{s}(\mathbb{R}^N)\cap L^{\frac{2Np}{N + \alpha}}(\mathbb{R}^N);\mathbb{R})$ defined by
\begin{equation*}\label{333e:functional}
E_{\alpha, p}(u) = \frac{1}{2}\displaystyle\int_{\mathbb{R}^N}|(-\Delta)^{\frac{s}{2}} u|^2 + |u|^2dx - \frac{1}{2p}\int_{\mathbb{R}^N}(1 + a(x))(I_{\alpha}\ast |u|^p)|u|^pdx
\end{equation*}
are weak solutions of (\ref{333e:main}). This functional is well defined by the Hardy-Littlewood-Sobolev inequality which states that if $t\in (1, \frac{N}{\alpha})$, then for every $v\in L^t(\mathbb{R}^N)$, $I_\alpha \ast v\in L^{\frac{Nt}{N - \alpha t}}(\mathbb{R}^N)$ and
\begin{equation*}
\displaystyle\int_{\mathbb{R}^N}|I_\alpha \ast v|^{\frac{Nt}{N - \alpha t}}dx\leq C\left(\int_{\mathbb{R}^N}|v|^tdx\right)^{\frac{N}{N - \alpha t}},
\end{equation*}
where $C > 0$ depends only on $\alpha$, $N$ and $t$. Note also that by the Sobolev embedding, $H^{s}(\mathbb{R}^N)\hookrightarrow L^{\frac{2Np}{N + \alpha}}(\mathbb{R}^N)$ if and only if $\frac{N - 2s}{N + \alpha} \leq \frac{1}{p} \leq \frac{N}{N + \alpha}$.

To prove Theorem \ref{333t:theorem1}, we use the idea of \cite{CeramiVairaJDE2010} which studied the positive solutions for some non-autonomous Schr\"{o}dinger-Poisson systems. Similar ideas were also used in \cite{AmbrosettiCeramiRuizSLCSSNAE2008}. The rest of this paper is organized as follows. In Section 2, we  study some properties of $E_{\alpha, p}$ under a natural constraint, the Nehari manifold. In Section 3, a crucial compactness theorem by the concentration compactness argument will be given. In Section 4, we prove Theorem \ref{333t:theorem1}.

\section{Preliminaries and Variational setting}
For $s\in (0, 1)$ and $N\geq 2$, the fractional Sobolev space $H^s(\mathbb{R}^N)$ can be defined by
\begin{equation*}
H^s(\mathbb{R}^N) = \displaystyle\left\{u\in L^2(\mathbb{R}^N): \frac{|u(x) - u(y)|}{|x -y|^{N/2 + s}}\in L^2(\mathbb{R}^N\times\mathbb{R}^N)\right\},
\end{equation*}
which is endowed with the norm
\begin{equation*}
\|u\|_{s} := \displaystyle\left(\int_{\mathbb{R}^N}|u|^2dx + \int_{\mathbb{R}^N}\int_{\mathbb{R}^N}\frac{|u(x) - u(y)|^2}{|x -y|^{N + 2s}}dxdy\right)^{\frac{1}{2}}.
\end{equation*}
The Gagliardo semi-norm of $u$ is defined by
\begin{equation*}
[u]_{H^s(\mathbb{R}^N)}  := \displaystyle\left(\int_{\mathbb{R}^N}\int_{\mathbb{R}^N}\frac{|u(x) - u(y)|^2}{|x -y|^{N + 2s}}dxdy\right)^{\frac{1}{2}}.
\end{equation*}
Let $\mathcal{S}$ be the Schwartz space of rapidly decaying smooth functions on $\mathbb{R}^{N}$ and the topology of $\mathcal{S}$ is generated by
$$
p_{m}(\varphi) = \displaystyle\sup_{x\in\mathbb{R}^{N}}(1 + |x|)^{m}\sum_{|\gamma|\leq m}|D^{\gamma}\varphi(x)|,\quad m = 0, 1, 2,\cdot\cdot\cdot,
$$
where $\varphi\in \mathcal{S}$. Denote the topological dual of $\mathcal{S}$ by $\mathcal{S}'$, then for any $\varphi\in \mathcal{S}$, the usual Fourier transformation of $\varphi$ is given by
$$
\mathcal{F}\varphi(\xi) = \frac{1}{(2\pi)^{N/2}}\int_{\mathbb{R}^{N}}e^{-i\xi\cdot x}\varphi(x)dx
$$
and one can extend $\mathcal{F}$ from $\mathcal{S}$ to $\mathcal{S}'$. Furthermore, it holds that
$$
[u]_s = C\left(\int_{\mathbb{R}^N}|\xi|^{2s}|\mathcal{F}u(\xi)|^2d\xi\right)^{\frac{1}{2}} = C\|(-\Delta)^{\frac{s}{2}}u\|_{L^2(\mathbb{R}^N)}
$$
for a suitable positive constant $C=C(N, \gamma)$. Hence we have
$$
\|u\|_s = \displaystyle\left(\int_{\mathbb{R}^N}|u|^2dx + \int_{\mathbb{R}^N}|(-\Delta)^{\frac{s}{2}}u|^2dx\right)^{\frac{1}{2}}.
$$
From the fractional Sobolev embedding theorem, $H^{s}(\mathbb R^N)$ embeds continuously into $L^q(\mathbb R^N)$ for all $q\in [2,\frac{2N}{N - 2s}]$ and compactly into $L_{{\rm loc}}^q(\mathbb R^N)$ for all $q\in [2,\frac{2N}{N - 2s})$, see \cite{DiNezzaPalatucciValdinoci2012}.

The functional $E_{\alpha, p}$ is bounded neither from below nor from above. So it is not convenient to consider $E_{\alpha, p}$ restricted to a natural constraint, the Nehari manifold, that contains the critical points of $E_{\alpha, p}$ and on which $E_{\alpha, p}$ turns out to be bounded from below. Define
\begin{equation}\label{333e:Nehari1}
\mathcal{N}:=\{u\in H^s(\mathbb{R}^N)\setminus\{0\}: G(u) = 0\},
\end{equation}
here
\begin{equation}\label{333e:Nehari2}
G(u) = E_{\alpha, p}'(u)[u] = \|u\|^2_s - \int_{\mathbb{R}^N}(1 + a(x))(I_{\alpha}\ast |u|^p)|u|^pdx.
\end{equation}
Note that
\begin{equation}\label{333e:Nehari3}
E_{\alpha, p}|_{\mathcal{N}}(u) = (\frac{1}{2} - \frac{1}{2p})\|u\|^2_s.
\end{equation}
\begin{lemma}\label{333l:nehari}
\begin{enumerate}
\item[(1)]$\mathcal{N}$ is a $C^1$ regular manifold which is diffeomorphic to the standard sphere of $W^{1, 2}(\mathbb{R}^N)$;
\item[(2)]$E_{\alpha, p}$ is bounded from below by a positive constant on $\mathcal{N}$;
\item[(3)]$u$ is a nonzero free critical point of $E_{\alpha,p}$ if and only if $u$ is a critical point of $E_{\alpha,p}$ constrained on $\mathcal{N}$.
\end{enumerate}
\end{lemma}
\begin{proof}
(1) Let $u\in H^s(\mathbb{R}^N)\setminus\{0\}$ with $\|u\|_s = 1$. Then there exists an unique $t\in \mathbb{R}^+\setminus\{0\}$ such that $tu\in \mathcal{N}$. Indeed, such $t$ must satisfy
\begin{equation*}
0 = E_{\alpha, p}'(tu)[tu] = t^2\|u\|^2_s - |t|^{2p}\int_{\mathbb{R}^N}(1 + a(x))(I_{\alpha}\ast |u|^p)|u|^pdx.
\end{equation*}
Define
$$
A: = \int_{\mathbb{R}^N}(1 + a(x))(I_{\alpha}\ast |u|^p)|u|^pdx,
$$
then we are led to find a positive solution of equation $t^2(1 - At^{2p - 2}) = 0$ with $A > 0$. Since $p > 1$, the equation $1 - At^{2p - 2} = 0$ has an unique solution $t = t(u) > 0$. The corresponding point $t(u)u\in \mathcal{N}$, which is called the projection of $u$ on $\mathcal{N}$. Moreover,
$$
E_{\alpha, p}(t(u)u) = \max_{t > 0} E_{\alpha, p}(tu).
$$

Suppose $u\in \mathcal{N}$, then
\begin{equation*}
\left|\int_{\mathbb{R}^N}(1 + a(x))(I_{\alpha}\ast |u|^p)|u|^pdx\right|\leq C\left|\int_{\mathbb{R}^N}(I_{\alpha}\ast |u|^p)|u|^pdx\right|\leq C\|u\|^{2p}
\end{equation*}
and then
\begin{eqnarray*}
0 &=& \|u\|^2_s - \int_{\mathbb{R}^N}(1 + a(x))(I_{\alpha}\ast |u|^p)|u|^pdx\\
&\geq & \|u\|^2_s - C\|u\|^{2p}_s.
\end{eqnarray*}
From which we have
$$
\|u\|_s\geq C_1 > 0\quad u\in \mathcal{N}.
$$

Since $E_{\alpha, p}$ a $C^2$ functional and
\begin{eqnarray*}
E''_{\alpha, p}[v, w]& = &\langle v, w\rangle - (p - 1)\int_{\mathbb{R}^N}(1 + a(x))(I_{\alpha}\ast |u|^p)|u|^{p - 2}vwdx\\ &&- p\int_{\mathbb{R}^N}\int_{\mathbb{R}^N}\frac{|u|^{p - 2}uw}{|x - y|^{N - \alpha}}dy(1 + a(x))|u|^{p - 1}vdx,
\end{eqnarray*}
$G$ is a $C^1$ functional and
\begin{eqnarray*}
G'(u)[u]& = & E''_{\alpha, p}[u, u]\\
 &=&\|u\|^2_s - (2p - 1)\int_{\mathbb{R}^N}(1 + a(x))(I_{\alpha}\ast |u|^p)|u|^{p}dx\\
 & = &(2 - 2p)\|u\|^2_s\leq (2 - 2p)C < 0.
\end{eqnarray*}

(2)
From the above argument, we have
\begin{equation*}
E_{\alpha, p}|_{\mathcal{N}}(u) = (\frac{1}{2} - \frac{1}{2p})\|u\|^2_s > C > 0.
\end{equation*}

(3) Clearly, if $u\neq 0$ is a critical point of $E_{\alpha, p}$, then $E_{\alpha, p}'(u) = 0$ and $u\in \mathcal{N}$. On the other hand, let $u$ be a critical point of $E_{\alpha, p}$ constrained on $\mathcal{N}$, then there exists $\lambda\in \mathbb{R}$ such that $E_{\alpha, p}'(u) = \lambda G'(u)$. Therefore,
$$
0 = G(u) = E_{\alpha, p}'(u)[u] = \lambda G'(u)[u].
$$
Since $G'(u)[u] < 0$, $\lambda = 0$ and $E_{\alpha,p}'(u) = 0$.
\end{proof}

Setting
$$
m:=\inf\{E_{\alpha, p}(u): u\in\mathcal{N}\},
$$
as a consequence of Lemma \ref{333l:nehari}, $m$ is a positive number.

When $a(x) \equiv 0$, equation (\ref{333e:main}) becomes
\begin{equation}\label{333e:infinity}
(- \Delta)^s u + u = (I_\alpha \ast |u|^{p})|u|^{p - 2}u\quad\text{ in }\mathbb{R}^N.
\end{equation}
In this case, we use the notation $E_{\alpha, p}^{\infty}(u)$ and $\mathcal{N}_{\infty}$, respectively, for the functional and the natural constraint. Namely,
\begin{equation*}\label{333e:functionalinfty}
E^{\infty}_{\alpha, p}(u) = \frac{1}{2}\displaystyle\int_{\mathbb{R}^N}|(-\Delta)^{\frac{s}{2}} u|^2 + |u|^2dx - \frac{1}{2p}\int_{\mathbb{R}^N}(I_{\alpha}\ast |u|^p)|u|^pdx,
\end{equation*}
\begin{equation*}\label{333e:Nehari1}
\mathcal{N}_{\infty}:=\{u\in H^{s}(\mathbb{R}^N)\setminus\{0\}: \|u\|^2_s - \int_{\mathbb{R}^N}(I_{\alpha}\ast |u|^p)|u|^pdx = 0\}.
\end{equation*}

In the following lemma, we state some known results about the existence of positive solutions of (\ref{333e:infinity}) which are useful in the following proof.
\begin{lemma}(\cite{AveniaSicilianoSquassinaMMMAS2015})
Let $N\in \mathbb{N}^+$, $s\in (0, 1)$, $\alpha\in (0, N)$ and $p\in (1, \infty)$. Assume that $\frac{N - 2s}{N + \alpha} < \frac{1}{p} < \frac{N}{N + \alpha}$, then equation (\ref{333e:infinity}) has a positive, ground state solution $w\in H^{s}(\mathbb{R}^N)$ which is radially symmetric about the origin and decaying to zero as $|x|\to +\infty$.
\end{lemma}

Since $w$ is the ground state solution, setting
$$
m_{\infty}:=\inf\{E_{\alpha, p}^\infty(u): u\in \mathcal{N}_{\infty}\},
$$
we have $E_{\alpha, p}^\infty(v)\geq E_{\alpha, p}^\infty(w)$, for all $v$ solution of (\ref{333e:infinity}). Note also that
$$
m_{\infty} = E_{\alpha, p}^{\infty}(w) = (\frac{1}{2} - \frac{1}{2p})\|w\|^2_s.
$$

\section{A compactness lemma}
In this section, we study the compactness of  the Palais-Smale sequence of $E_{\alpha, p}$. We follow the ideas of \cite{CeramiVairaJDE2010}.
\begin{theorem}\label{333L:compact}
Let $\{u_n\}$ be a Palais-Smale sequence of $E_{\alpha, p}$ constrained on $\mathcal {N}$, that is to say, $u_n\in \mathcal{N}$ and
\begin{equation*}
E_{\alpha, p}(u_n) \text{ is bounded},\quad E'_{\alpha,p}|_{\mathcal{N}}(u_n)\to 0\text{ strongly in } H^s(\mathbb{R}^N).
\end{equation*}
Then, up to a subsequence, there exist a solution $\bar{u}$ of (\ref{333e:main}), a number $k\in\mathbb{N}\cup \{0\}$, $k$ functions $u^1,\cdot\cdot\cdot, u^k$ of $H^{s}(\mathbb{R}^N)$ and $k$ sequence of points $\{y_n^i\}$, $y_n^i\in \mathbb{R}^N$, $0\leq i\leq k$, such that
\begin{enumerate}
\item[(i)] $$|y_n^i|\to \infty, \quad |y_n^i-y_n^j|\to \infty,\quad \mbox{if} \quad i\ne j, \quad n\to \infty;$$
\item[(ii)]
\begin{equation*}
u_n-\sum_{i=1}^{k}u_n^i(\cdot - y_n^i) \to \bar{u}\text{ in }H^s(\mathbb{R}^N);
\end{equation*}
\item[(iii)]
$$E_{\alpha, p}(u_n)\to \sum_{i=1}^{k}E_{\alpha, p}^\infty(u^i) + E_{\alpha, p}(\bar{u});$$
\item[(iv)] $u^i$ are nontrivial weak solutions of (\ref{333e:infinity}).
\end{enumerate}
Here, we agree that in the case $k = 0$ the above holds without $u^i$s.
\end{theorem}
\begin{proof}
Since $E_{\alpha, p}(u_n)$ is bounded, from the fact that
\begin{equation*}
E_{\alpha, p}|_{\mathcal{N}}(u) = (\frac{1}{2} - \frac{1}{2p})\|u\|_s^2,
\end{equation*}
we have that $\{u_n\}_n$ is bounded, too.

Now, we claim that
$$
E'_{\alpha, p}(u_n) \to 0 \quad\text{ in } H^s(\mathbb{R}^N).
$$
In fact,we have
$$
o(1) = E'_{\alpha, p}|_{\mathcal{N}}(u_n) = E'_{\alpha, p}(u_n) - \lambda_n G'(u_n)
$$
for some $\lambda_n\in \mathbb{R}$. Taking the scalar product with $u_n$, we obtain
$$
o(1) = \langle E'_{\alpha, p}(u_n), u_n\rangle - \lambda_n \langle G'(u_n), u_n\rangle.
$$
Since $u_n\in \mathcal{N}$, $\langle E'_{\alpha, p}(u_n), u_n\rangle = 0$ and $\langle G'(u_n), u_n\rangle < C < 0$. Thus $\lambda_n\to 0$ for $n\to+\infty$. Moreover, by the boundedness of $\{u_n\}$, $G'(u_n)$ is bounded and this implies that $\lambda_n G'(u_n)\to 0$, so we have the assertion.

On the other hand, since $u_n$ is bounded in $H^s(\mathbb{R}^N)$, there exists $\bar{u}\in H^s(\mathbb{R}^N)$ such that, up to a subsequence,
$$
u_n\rightharpoonup \bar{u}\quad\text{ in }H^s(\mathbb{R}^N)\text{ and in } L^{\frac{2Np}{N + \alpha}}(\mathbb R^N),
$$
$$
u_n(x)\to \bar{u}(x) \text{ a.e. on }\mathbb{R}^N.
$$
Thus we easily duce that $E'_{\alpha, p}(\bar{u}) = 0$, that is to say, $\bar{u}$ is a weak solution of (\ref{333e:main}). Indeed, for any smooth function $h$ with compact support $\Omega\subseteq \mathbb{R}^N$,
\begin{equation}
\begin{aligned}
&\left|\int_{\mathbb{R}^N}(1 + a(x))(I_{\alpha}\ast |u_n|^p)|u_n|^{p - 2}u_nhdx - \int_{\mathbb{R}^N}(1 + a(x))(I_{\alpha}\ast |\bar{u}|^p)|\bar{u}|^{p - 2}\bar{u}hdx\right|\\
&\leq \left|\int_{\Omega}(1 + a(x))(I_{\alpha}\ast |u_n|^p)|u_n|^{p - 2}u_nhdx - \int_{\Omega}(1 + a(x))(I_{\alpha}\ast |u_n|^p)|\bar{u}|^{p - 2}\bar{u}hdx\right|\\
&\quad +   \left|\int_{\Omega}(1 + a(x))(I_{\alpha}\ast |u_n|^p)|\bar{u}|^{p - 2}\bar{u}hdx - \int_{\Omega}(1 + a(x))(I_{\alpha}\ast |\bar{u}|^p)|\bar{u}|^{p - 2}\bar{u}hdx\right|\\
&= \left|\int_{\Omega}(1 + a(x))(I_{\alpha}\ast |u_n|^p)(|u_n|^{p - 2}u_n - |\bar{u}|^{p - 2}\bar{u})hdx\right|\\
&\quad + \left|\int_{\Omega}(1 + a(x))(I_{\alpha}\ast |u_n|^p - I_{\alpha}\ast a|\bar{u}|^p)|\bar{u}|^{p - 2}\bar{u}hdx\right|\\
&\leq C\left|\int_{\Omega}(I_{\alpha}\ast |u_n|^p)(|u_n|^{p - 2}u_n - |\bar{u}|^{p - 2}\bar{u})hdx\right| + C\left|\int_{\Omega}\left(I_{\alpha}\ast (|u_n|^p - |\bar{u}|^p)\right)|\bar{u}|^{p - 2}\bar{u}hdx\right|\\
&\leq C\left|I_{\alpha}\ast |u_n|^p\right|_{L^{\frac{2N}{N - \alpha}}(\mathbb{R}^N)}\left||u_n|^{p - 2}u_n - |\bar{u}|^{p - 2}\bar{u}\right|_{L^{\frac{2Np}{(p - 1)(N + \alpha)}}(\Omega)}|h|_{L^{\frac{2Np}{N + \alpha}}(\mathbb{R}^N)}\\
&+ C\left|I_{\alpha}\ast (|u_n|^p - |\bar{u}|^p)\right|_{L^{\frac{2N}{N - \alpha}}(\Omega)}|\bar{u}|^{p - 1}_{L^{\frac{2Np}{N + \alpha}}(\mathbb{R}^N)}|h|_{L^{\frac{2Np}{N + \alpha}}(\mathbb{R}^N)}\\
&\leq C\|u_n\|^p\left||u_n|^{p - 2}u_n - |\bar{u}|^{p - 2}\bar{u}\right|_{L^{\frac{2Np}{(p - 1)(N + \alpha)}}(\Omega)}\|h\| + C\left||u_n|^p - |\bar{u}|^p\right|_{L^{\frac{2N}{N + \alpha}}(\Omega)}\|\bar{u}\|^{p - 1}\|h\|.
\end{aligned}
\end{equation}
Since $H^s(\mathbb{R}^N)\hookrightarrow L_{loc}^{\frac{2Np}{N + \alpha}}(\mathbb R^N)$ compactly, we have $|u_n - \bar{u}|_{L^{\frac{2Np}{N + \alpha}}(\Omega)}\to 0$. By Lemma 1.20 in \cite{zouwenmingschechter}, $\left||u_n|^{p - 2}u_n - |\bar{u}|^{p - 2}\bar{u}\right|_{L^{\frac{2Np}{(p - 1)(N + \alpha)}}(\Omega)}\to 0$ and $||u_n|^p - |\bar{u}|^p |_{L^{\frac{2N}{N + \alpha}}(\Omega)}\to 0$. So $E'_{\alpha, p}(\bar{u}) = 0$.

If $u_n \to \bar{u}$ in $H^s(\mathbb{R}^N)$, we are done. So we can assume that $\{u_n\}$ does not converge strongly to $\bar{u}$ in $H^s(\mathbb{R}^N)$. Set
$$
z_n^1 = u_n(x) - \bar{u}(x).
$$

Obviously, we have $z_n^1\rightharpoonup 0$ in $H^s(\mathbb{R}^N)$, but not strongly. Then by direct computation we have
\begin{equation}\label{333E:B5}
\|u_n\|^2=\|z_n^1 + \bar{u}\|^2=\|z_n^1\|^2+\|\bar{u}\|^2+o(1).
\end{equation}

Now, we claim that
\begin{equation}\label{333e:estimate4}
\begin{aligned}
\displaystyle\int_{\mathbb{R}^N}a(x)(I_{\alpha}\ast |u_n|^p)(x)|u_n(x)|^{p - 2}u_n(x) & h(x)dx\\ &= \int_{\mathbb{R}^N}a(x)(I_{\alpha}\ast |\bar{u}|^p)(x)|\bar{u}(x)|^{p - 2}\bar{u}(x) h(x)dx + o(1),
\end{aligned}
\end{equation}
\begin{equation}\label{333e:estimate5}
\displaystyle\int_{\mathbb{R}^N}a(x)(I_{\alpha}\ast |u_n|^p)(x)|u_n(x)|^pdx = \int_{\mathbb{R}^N}a(x)(I_{\alpha}\ast |\bar{u}|^p)(x)|\bar{u}(x)|^pdx + o(1),
\end{equation}
\begin{equation}\label{333e:estimate7}
\begin{aligned}
\displaystyle\int_{\mathbb{R}^N}(I_{\alpha}\ast |u_n|^p)(x)&|u_n(x)|^{p - 2}u_n(x) h(x)dx
= \int_{\mathbb{R}^N}(I_{\alpha}\ast |\bar{u}|^p)(x)|\bar{u}(x)|^{p - 2}\bar{u}(x) h(x)dx\\
&+ \int_{\mathbb{R}^N}(I_{\alpha}\ast |z_n^1|^p)(x)|z_n^1(x)|^{p - 2}z_n^1(x) h(x)dx + o(1).
\end{aligned}
\end{equation}

Let us observe that in view of the Sobolev embedding theorems, $u_n\rightharpoonup \bar{u}$ in $H^s(\mathbb{R}^N)$ implies
\begin{equation*}
u_n\rightharpoonup \bar{u}\quad\text{in }L^{\frac{2N}{N - 2s}}(\mathbb{R}^N),\quad |u_n|^p\to |\bar{u}|^p\quad\text{in }L^{\frac{2N}{p(N - 2s)}}_{loc}(\mathbb{R}^N),
\end{equation*}
\begin{equation*}
\phi_{u_n}\rightharpoonup\phi_{\bar{u}}\quad\text{in } \dot{H}^{\frac{\alpha}{2}}(\mathbb{R}^N), \quad\phi_{u_n}\to\phi_{\bar{u}}\quad\text{in } L^{\frac{2N}{N - \alpha}}_{loc}(\mathbb{R}^N).
\end{equation*}

{\it Proof of} (\ref{333e:estimate4}). We estimate
\begin{equation*}
\begin{aligned}
&\left|\displaystyle\int_{\mathbb{R}^N}a(x)(I_{\alpha}\ast |u_n|^p)(x)|u_n(x)|^{p - 2}u_n(x) h(x)dx - \int_{\mathbb{R}^N}a(x)(I_{\alpha}\ast |\bar{u}|^p)(x)|\bar{u}(x)|^{p - 2}\bar{u}(x) h(x)dx\right|\\
&\leq \left|\displaystyle\int_{\mathbb{R}^N}a(x)\left((I_{\alpha}\ast |u_n|^p)(x) - (I_{\alpha}\ast |\bar{u}|^p)(x)\right)|u_n(x)|^{p - 2}u_n(x) h(x)dx\right|\\
& + \left|\int_{\mathbb{R}^N}a(x)(I_{\alpha}\ast |\bar{u}|^p)(x)\left(|u_n(x)|^{p - 2}u_n(x)- |\bar{u}(x)|^{p - 2}\bar{u}(x)\right)h(x)dx\right|.
\end{aligned}
\end{equation*}
Since
\begin{equation*}
\begin{aligned}
&\left|\int_{\mathbb{R}^N}a(x)(I_{\alpha}\ast |\bar{u}|^p)(x)\left(|u_n(x)|^{p - 2}u_n(x)- |\bar{u}(x)|^{p - 2}\bar{u}(x)\right)h(x)dx\right|\\
&\leq C\left|a I_{\alpha}\ast |\bar{u}|^p\right|_{L^{\frac{2N}{N - \alpha}}(\mathbb{R}^N)}\left||u_n|^{p - 2}u_n - |\bar{u}|^{p - 2}\bar{u}\right|_{L^{\frac{2Np}{(p - 1)(N + \alpha)}}(B_\rho(0))}|h|_{L^{\frac{2Np}{N + \alpha}}(\mathbb{R}^N)}\\
&+ C\left|a I_{\alpha}\ast |\bar{u}|^p\right|_{L^{\frac{2N}{N - \alpha}}(\mathbb{R}^N\setminus B_\rho(0))}\left||u_n|^{p - 2}u_n - |\bar{u}|^{p - 2}\bar{u}\right|_{L^{\frac{2Np}{(p - 1)(N + \alpha)}}(\mathbb{R}^N)}|h|_{L^{\frac{2Np}{N + \alpha}}(\mathbb{R}^N)}\\
&\leq C\left|aI_{\alpha}\ast |\bar{u}|^p\right|_{L^{\frac{2N}{N - \alpha}}(\mathbb{R}^N)}\left||u_n|^{p - 2}u_n - |\bar{u}|^{p - 2}\bar{u}\right|_{L^{\frac{2Np}{(p - 1)(N + \alpha)}}(B_\rho(0))}|h|_{L^{\frac{2Np}{N + \alpha}}(\mathbb{R}^N)}\\
&+ C|a|_{L^{\frac{2N}{N - Np + 2sp + \alpha}}(\mathbb{R}^N\setminus B_\rho(0))}\|\bar{u}\|^p\left||u_n|^{p - 2}u_n - |\bar{u}|^{p - 2}\bar{u}\right|_{L^{\frac{2Np}{(p - 1)(N + \alpha)}}(\mathbb{R}^N)}|h|_{L^{\frac{2Np}{N + \alpha}}(\mathbb{R}^N)}\\
&\leq C\varepsilon,
\end{aligned}
\end{equation*}
we have the following estimate,
\begin{equation*}
\begin{aligned}
&\left|\displaystyle\int_{\mathbb{R}^N}a(x)(I_{\alpha}\ast |u_n|^p)(x)|u_n(x)|^{p - 2}u_n(x) h(x)dx - \int_{\mathbb{R}^N}a(x)(I_{\alpha}\ast |\bar{u}|^p)(x)|\bar{u}(x)|^{p - 2}\bar{u}(x) h(x)dx\right|\\
&\leq \left|\displaystyle\int_{\mathbb{R}^N}a(x)\left((I_{\alpha}\ast |u_n|^p)(x) - (I_{\alpha}\ast |\bar{u}|^p)(x)\right)|u_n(x)|^{p - 2}u_n(x) h(x)dx\right|+ C\varepsilon\\
&\leq C|a|_{L^{\frac{2N}{N - Np + 2sp + \alpha}}(\mathbb{R}^N\setminus B_\rho(0))}||u_n|^p - |\bar{u}|^p|_{L^{\frac{2N}{p(N - 2s)}}(\mathbb{R}^N)}|u_n|^{p - 1}_{L^{\frac{2Np}{N + \alpha}}(\mathbb{R}^N)}|h|_{L^{\frac{2Np}{N + \alpha}}(\mathbb{R}^N)}\\
& + C|a|_{L^{\frac{2N}{N - Np + 2sp + \alpha}}(\mathbb{R}^N)}||u_n|^p - |\bar{u}|^p|_{L^{\frac{2N}{p(N - 2s)}}(B_\rho(0))}|u_n|^{p - 1}_{L^{\frac{2Np}{N + \alpha}}(\mathbb{R}^N)}|h|_{L^{\frac{2Np}{N + \alpha}}(\mathbb{R}^N)} + C\varepsilon
\leq  C\varepsilon.
\end{aligned}
\end{equation*}
Then we have (\ref{333e:estimate4}).

{\it Proof of }(\ref{333e:estimate5}). We estimate
\begin{equation*}
\begin{aligned}
&\left|\displaystyle\int_{\mathbb{R}^N}a(x)(I_{\alpha}\ast |u_n|^p)(x)|u_n(x)|^pdx - \int_{\mathbb{R}^N}a(x)(I_{\alpha}\ast |\bar{u}|^p)(x)|\bar{u}(x)|^pdx \right|\\
&\leq \left|\displaystyle\int_{\mathbb{R}^N}a(x)(I_{\alpha}\ast|u_n|^p)(x)\left(|u_n(x)|^p - |\bar{u}(x)|^p\right)dx\right|\\
&+ \left|\int_{\mathbb{R}^N}a(x)\left((I_{\alpha}\ast |u_n|^p)(x) - (I_{\alpha}\ast |\bar{u}|^p)(x)\right)|\bar{u}(x)|^pdx \right|.
\end{aligned}
\end{equation*}
The first term satisfies
\begin{equation*}
\begin{aligned}
&\left|\displaystyle\int_{\mathbb{R}^N}a(x)(I_{\alpha}\ast|u_n|^p)(x)\left(|u_n(x)|^p - |\bar{u}(x)|^p\right)dx\right|\\
&\leq C|a|_{L^{\frac{2N}{N - Np + 2sp + \alpha}}(\mathbb{R}^N\setminus B_\rho(0))}||u_n|^p - |\bar{u}|^p|_{L^{\frac{2N}{p(N - 2s)}}(\mathbb{R}^N)}\|u_n\|^p\\
& + C|a|_{L^{\frac{2N}{N - Np + 2sp + \alpha}}(\mathbb{R}^N)}||u_n|^p - |\bar{u}|^p|_{L^{\frac{2N}{p(N - 2s)}}(B_\rho(0))}\|u_n\|^p\\
&\leq C\varepsilon.
\end{aligned}
\end{equation*}
The second term satisfies
\begin{equation*}
\begin{aligned}
&\left|\int_{\mathbb{R}^N}a(x)\left((I_{\alpha}\ast |u_n|^p)(x) - (I_{\alpha}\ast |\bar{u}|^p(x))\right)|\bar{u}(x)|^pdx \right|\\
&\leq C|a|_{L^{\frac{2N}{N - Np + 2sp + \alpha}}(\mathbb{R}^N\setminus B_\rho(0))}|\bar{u}|^p_{L^{\frac{2N}{N - 2s}}(\mathbb{R}^N)}|(I_{\alpha}\ast |u_n|^p) - (I_{\alpha}\ast |\bar{u})|^p|_{L^{\frac{2N}{N - \alpha}}(\mathbb{R}^N)}\\
& + C|a|_{L^{\frac{2N}{N - Np + 2sp + \alpha}}(\mathbb{R}^N)}|\bar{u}|^p_{L^{\frac{2N}{N - 2s}}(\mathbb{R}^N)}|(I_{\alpha}\ast |u_n|^p) - (I_{\alpha}\ast |\bar{u})|^p|_{L^{\frac{2N}{N - \alpha}}(B_\rho(0))}\\
&\leq C|a|_{L^{\frac{2N}{N - Np + 2sp + \alpha}}(\mathbb{R}^N\setminus B_\rho(0))}|\bar{u}|^p_{L^{\frac{2N}{N - 2s}}(\mathbb{R}^N)}|(I_{\alpha}\ast |u_n|^p) - (I_{\alpha}\ast |\bar{u})|^p|_{L^{\frac{2N}{N - \alpha}}(\mathbb{R}^N)}\\
& + C|a|_{L^{\frac{2N}{N - Np + 2sp + \alpha}}(\mathbb{R}^N)}|\bar{u}|^p_{L^{\frac{2N}{N - 2s}}(\mathbb{R}^N)}||u_n|^p - |\bar{u}|^p|_{L^{\frac{2N}{N + \alpha}}(B_\rho(0))}\\
&\leq C\varepsilon.
\end{aligned}
\end{equation*}
Thus we have (\ref{333e:estimate5}).

{\it Proof of }(\ref{333e:estimate7}). Indeed, we have
\begin{equation*}
\begin{aligned}
&\displaystyle\int_{\mathbb{R}^N}(I_{\alpha}\ast |u_n|^p)(x)|u_n(x)|^{p - 2}u_n(x) h(x)dx\\ &-
\int_{\mathbb{R}^N}(I_{\alpha}\ast |u_n - \bar{u}|^p)(x)|u_n(x) - \bar{u}(x)|^{p - 2}(u_n - \bar{u})(x) h(x)dx\\
&= \displaystyle\int_{\mathbb{R}^N}(I_{\alpha}\ast (|u_n|^p - |u_n - \bar{u}|^p))(x)\Big(|u_n(x)|^{p - 2}u_n(x)dx\\
&- \displaystyle\int_{\mathbb{R}^N}|u_n(x) - \bar{u}(x)|^{p - 2}(u_n - \bar{u})(x)\Big) h(x)dx\\
&+\displaystyle\int_{\mathbb{R}^N}(I_{\alpha}\ast |u_n|^p)(x)|u_n(x) - \bar{u}(x)|^{p - 2}(u_n - \bar{u})(x) h(x)dx
\end{aligned}
\end{equation*}
\begin{equation*}
\begin{aligned}
&+ \displaystyle\int_{\mathbb{R}^N}(I_{\alpha}\ast |u_n - \bar{u}|^p)(x)|u_n(x)|^{p - 2}u_n(x) h(x)dx \\
& - 2\displaystyle\int_{\mathbb{R}^N}(I_{\alpha}\ast|u_n - \bar{u}|^p)(x)|u_n(x) - \bar{u}(x)|^{p - 2}(u_n - \bar{u})(x) h(x)dx.
\end{aligned}
\end{equation*}
From the proof of Lemma 4.3 in \cite{AveniaSicilianoSquassinaMMMAS2015}, we have $I_{\alpha}\ast (|u_n|^p - |u_n - \bar{u}|^p)\to I_{\alpha}\ast |u|^p$ in $L^{\frac{2N}{N - \alpha}}(\mathbb{R}^N)$ and $|u_n - \bar{u}|^p\to 0$ weakly in $L^{\frac{2N}{N + \alpha}}(\mathbb{R}^N)$ as $n\to\infty$. So we have (\ref{333e:estimate7}).

Combining the above estimates and Lemma 4.3 in \cite{AveniaSicilianoSquassinaMMMAS2015}, we obtain
\begin{equation*}
\begin{aligned}
E_{\alpha, p}(u_n) &= \frac{1}{2}\displaystyle\int_{\mathbb{R}^N}|\nabla u_n|^2 + |u_n|^2dx - \frac{1}{2p}\int_{\mathbb{R}^N}(1 + a(x))(I_{\alpha}\ast |u_n|^p)|u_n|^pdx\\
& = \frac{1}{2}\displaystyle\int_{\mathbb{R}^N}|\nabla u_n|^2 + |u_n|^2dx - \frac{1}{2p}\int_{\mathbb{R}^N}a(x)(I_{\alpha}\ast |u_n|^p)(x)|u_n(x)|^pdx \\& - \frac{1}{2p}\int_{\mathbb{R}^N}(I_{\alpha}\ast |u_n|^p)(x)|u_n(x)|^pdx\\
& = \frac{1}{2}\displaystyle\int_{\mathbb{R}^N}|\nabla u_n|^2 + |u_n|^2dx - \frac{1}{2p}\int_{\mathbb{R}^N}a(x)(I_{\alpha}\ast |\bar{u}|^p)(x)|\bar{u}(x)|^pdx \\ & - \frac{1}{2p}\int_{\mathbb{R}^N}(I_{\alpha}\ast |\bar{u}|^p)(x)|\bar{u}(x)|^pdx -\frac{1}{2p}\int_{\mathbb{R}^N}(I_{\alpha}\ast |z_n^1|^p)(x)|z_n^1(x)|^pdx + o(1)\\
&=\frac{1}{2}\|z_n^1\|^2+\frac{1}{2}\|\bar{u}\|^2 -\frac{1}{2p}\int_{\mathbb{R}^N}(1 + a(x))(I_{\alpha}\ast |\bar{u}|^p)(x)|\bar{u}(x)|^pdx\\ &-\frac{1}{2p}\int_{\mathbb{R}^N}(I_{\alpha}\ast |z_n^1|^p)(x)|z_n^1(x)|^pdx + o(1)\\
&=E_{\alpha, p}(\bar{u})+E_{\alpha, p}^\infty(z_n^1)+o(1),
\end{aligned}
\end{equation*}
and for all $h\in H^s(\mathbb{R}^N)$,
\begin{eqnarray*}\label{333E:pss1}
o(1)&=&(E'_{\alpha, p}(u_n), h)\\
&=& \langle u_n, h\rangle- \int_{\mathbb{R}^N}(1 + a(x))(I_{\alpha}\ast |u_n|^p)(x)|u_n(x)|^{p - 2}u_n(x) hdx\\
&=& \langle u_n, h\rangle- \int_{\mathbb{R}^N}a(x)(I_{\alpha}\ast |u_n|^p)(x)|u_n(x)|^{p - 2}u_n(x) h(x)dx\\
&&-\displaystyle\int_{\mathbb{R}^N}(I_{\alpha}\ast |u_n|^p)(x)|u_n(x)|^{p - 2}u_n(x) h(x)dx.
\end{eqnarray*}
Thus
\begin{eqnarray*}\label{333E:pss1}
o(1)&=&(E'_{\alpha, p}(u_n), h)\\
&=& \langle \bar{u}, h\rangle + \langle z_n^1, h\rangle- \int_{\mathbb{R}^N}a(x)(I_{\alpha}\ast |\bar{u}|^p)(x)|\bar{u}(x)|^{p - 2}\bar{u}(x) h(x)dx
\\&& - \int_{\mathbb{R}^N}(I_{\alpha}\ast |\bar{u}|^p)(x)|\bar{u}(x)|^{p - 2}\bar{u}(x) h(x)dx
\end{eqnarray*}
\begin{eqnarray*}
&& - \int_{\mathbb{R}^N}(I_{\alpha}\ast |z_n^1|^p)(x)|z_n^1(x)|^{p - 2}z_n^1(x)h dx + o(1)\\
&=& \langle \bar{u}, h\rangle + \langle z_n^1, h\rangle- \int_{\mathbb{R}^N}(1 + a(x))(I_{\alpha}\ast |\bar{u}|^p)(x)|\bar{u}(x)|^{p - 2}\bar{u}(x) h(x)dx\\
&& - \int_{\mathbb{R}^N}(I_{\alpha}\ast |z_n^1|^p)(x)|z_n^1(x)|^{p - 2}z_n^1(x)h(x)dx + o(1)\\
& = & \langle E'_{\alpha, p}(\bar{u}), h\rangle + \langle {E_{\alpha, p}^{\infty}}'(z_n^1), h\rangle + o(1)\\
& = & \langle {E_{\alpha, p}^{\infty}}'(z_n^1), h\rangle + o(1).
\end{eqnarray*}

Hence
\begin{equation}\label{333E:pss}
{E_{\alpha, p}^{\infty}}'(z_n^1) = o(1)\quad \mbox{in }\in H^s(\mathbb{R}^N).
\end{equation}
Furthermore,
\begin{eqnarray*}
0 &=& \langle E'_{\alpha, p}(u_n), u_n\rangle = \langle E'_{\alpha, p}(\bar{u}), \bar{u}\rangle + \langle {E_{\alpha, p}^{\infty}}'(z_n^1), z_n^1\rangle + o(1)\\
&=& \langle {E_{\alpha, p}^{\infty}}'(z_n^1), z_n^1\rangle + o(1).
\end{eqnarray*}

Setting
$$\delta:=\lim_{n\to +\infty}\sup\left(\sup_{y\in\mathbb R^N}\int_{B_1(y)}|z_n^1|^{\frac{2Np}{N + \alpha}}dx\right),$$
we have that $\delta>0$. Otherwise, if $\delta=0$, then by \cite[Lemma 1.21]{Wi:MT}, $z_n^1\to 0$ in $L^{\frac{2Np}{N + \alpha}}(\mathbb R^N)$. This is a contradiction to the fact that $u_n$ does not converge strongly to $\bar{u}$ in $L^{\frac{2Np}{N + \alpha}}(\mathbb R^N)$.

Then we may assume there exists a sequence of $\{y_n^1\}\subset \mathbb R^N$ such that
$$\int_{B_1(y_n^1)}|z_n^1|^{\frac{2Np}{N + \alpha}}dx>\delta/2.$$
Now we consider $z_n^1(\cdot + y_n^1)$. We assume $z_n^1(\cdot + y_n^1)\rightharpoonup u^1$ in $H^s(\mathbb{R}^N)$. Therefore, $z_n^1(\cdot + y_n^1)\to u^1$ a.e. on $\mathbb R^N$. Since
$$\int_{B_1(0)}|z_n^1(x + y_n^1)|^{\frac{2Np}{N + \alpha}}dx>\delta/2,$$
from the Rellich theorem it follows that
$$\int_{B_1(0)}|u^1(x)|^{\frac{2Np}{N + \alpha}}dx > \delta/2.$$
Thus, $u^1\ne 0$. Since $z_n^1\rightharpoonup 0$ in $W^{1,2}(\mathbb{R}^N)$, $(y_n^1)$ must be unbounded, and up to a subsequence, we can assume that $|y_n^1|\to +\infty$. Furthermore, (\ref{333E:pss}) implies ${E_{\alpha, p}^{\infty}}'(u^1) = 0$. Finally, let us set
$$
z_n^2(x) = z_n^1(x) - u^1(x - y^1_n).
$$

Then, we have
$$
\|z_n^2\|^2 = \|u_n\|^2  - \|u^1\|^2 -\|\bar{u}\|^2+o(1)
$$
and
\begin{eqnarray*}
\int_{\mathbb{R}^N}(I_{\alpha}\ast |z_n^2|^p)(x)|z_n^2(x)|^pdx &=& \displaystyle\int_{\mathbb{R}^N}(I_{\alpha}\ast |u_n|^p)(x)|u_n(x)|^pdx
-\int_{\mathbb{R}^N}(I_{\alpha}\ast |\bar{u}|^p)(x)|\bar{u}(x)|^pdx\\ &&- \int_{\mathbb{R}^N}(I_{\alpha}\ast |u^1|^p)(x)|u^1(x)|^pdx + o(1).
\end{eqnarray*}
This implies
$$
E_{\alpha, p}^\infty(z_n^2) = E_{\alpha, p}^\infty(z_n^1) - E_{\alpha, p}^\infty(u^1) + o(1),
$$
and hence we obtain
\begin{eqnarray*}
E_{\alpha, p}(u_n) &=&E_{\alpha, p}(\bar{u})+ E_{\alpha, p}^\infty(z_n^2) + E_{\alpha, p}^\infty(u^1)+o(1).
\end{eqnarray*}
As before, one can prove that
\begin{equation*}
{E_{\alpha, p}^{\infty}}'(z_n^2) = o(1)\quad \mbox{in } H^s(\mathbb{R}^N).
\end{equation*}

Now, if $z_n^2\to 0$ in $H^s(\mathbb{R}^N)$, then we are done. Otherwise, $z_n^2\rightharpoonup 0$ and not strongly and we repeat the above argument. Then we obtain a sequence of points $\{y_n^j\} \subseteq \mathbb{R}^N$ such that $|y_n^j|\to\infty$, $|y_n^j - y_n^i|\to+\infty$ if $i\neq j$ as $n\to+\infty$ and a sequence of functions $z_n^j(x) = z_n^{j - 1} - u^{j - 1}(x - y_n^{j - 1})$ with $j\geq 2$ such that
$$
z_n^j(x + y_n^j)\rightharpoonup u^j(x)\quad\text{ in }H^s(\mathbb{R}^N),\quad {E_{\alpha, p}^{\infty}}'(u^j) = 0
$$
and
$$
E_{\alpha, p}(u_n) = E_{\alpha, p}(\bar{u}) + \sum_{j = 1}^{k-1} E_{\alpha, p}^\infty(u^j) + E_{\alpha, p}^\infty(z_n^k) + o(1).
$$

Then, since $E_{\alpha,p}^\infty(u^j)\geq m_\infty$ for all $j$ and $E_{\alpha, p}(u_n)$ is bounded, the iteration must stop at some finite index $k$.
Thus we have completed the proof of Theorem \ref{333L:compact}.
\end{proof}

\begin{lemma}\label{333l:compactnesscorrollary}
Let $\{u_n\}$ be a $(PS)_d$ sequence. Then $\{u_n\}$ is relatively compact for all $d\in (0, m_\infty)$. Moreover, if $E_{\alpha, p}(u_n)\to m_\infty$, then either $\{u_n\}$ is relatively compact or the statement of Thoerem \ref{333L:compact} holds with $k = 1$, and $u^1 = w$, the ground state solution of (\ref{333e:infinity}).
\end{lemma}
\begin{proof}
Let us consider a $(PS)_d$ sequence $\{u_n\}$ and apply Theorem \ref{333L:compact}.  Specially, note that $E_{\alpha, p}^\infty(u^j) \geq m_\infty$, for all $j$.

When $E_{\alpha, p}(u_n) \to d < m_\infty$, then $k = 0$, and then $u_n\to \bar{u}$ in $H^s(\mathbb{R}^N)$. When $E_{\alpha, p}(u_n)\to m_\infty$, if $\{u_n\}$ is not compact, then $k = 1$, and $\bar{u} = 0$, $u^1 = w$.
\end{proof}

\section{Proof of Theorem \ref{333t:theorem1}}
\begin{proof}
To prove the existence of a ground state solution of (\ref{333e:main}), we just need to show that
\begin{equation}\label{333e:condition}
m < m_\infty.
\end{equation}
If this is the case, using Lemma \ref{333l:compactnesscorrollary} and standard arguments, it is easy to see that $m$ is achieved by a function $u$ which solves (\ref{333e:main}). Furthermore, $u$ is positive. Indeed, let $\{u_n\}\subseteq \mathcal{N}$ be a minimizing sequence, $E_{\alpha, p}(u_n)\to m$. By Theorem 6.17 of \cite{LiebLoss2001}, we have $\||u_n|\|^2 \leq \|u_n\|^2 = \int_{\mathbb{R}^N}(1 + a(x))(I_{\alpha}\ast |u_n|^p)|u_n|^pdx$. So $t_n|u_n|\in \mathcal{N}$ for some $t_n\in (0, 1]$. Thus by (\ref{333e:Nehari3}), we have $E_{\alpha, p}(t_n|u_n|) = (\frac{1}{2} - \frac{1}{2p})t_n^2\||u_n|\|^2\leq (\frac{1}{2} - \frac{1}{2p})\|u_n\|^2 = E_{\alpha, p}(u_n)$. This shows that $\{t_n|u_n|\}$ is also a minimizing sequence and the minimizer $u \geq 0$. By minor modification of Theorem 3.2 in \cite{AveniaSicilianoSquassinaMMMAS2015}, $u\in C^0(\mathbb{R}^N)$. Finally, from the maximum principle for fractional Laplacian (see \cite{DiNezzaPalatucciValdinoci2012}), we have $u > 0$.

To verify condition (\ref{333e:condition}), we consider the projection $tw$ on $\mathcal N$ of the minimizer $w$ of $E_{\alpha, p}^\infty$ on $\mathcal{N}_\infty$.
First, let us show that $t < 1$. Indeed, if $t\geq 1$ would be true, then we have
\begin{eqnarray*}
0 &=& t^2\|w\|^2 - t^{2p}\int_{\mathbb{R}^N}(1 + a(x))(I_{\alpha}\ast |w|^p)|w|^pdx \\
  & < &t^2\|w\|^2 - t^{2p}\int_{\mathbb{R}^N}(I_{\alpha}\ast |w|^p)|w|^pdx \\
  &=& (t^2 - t^{2p})\|w\|^2 \leq 0,
\end{eqnarray*}
a contradiction. So $t < 1$.

Then we have
\begin{eqnarray*}
m &\leq & E_{\alpha, p}(tw) =  \frac{1}{2}t^2\|w\|^2 - \frac{t^{2p}}{2p}\int_{\mathbb{R}^N}(1 + a(x))(I_{\alpha}\ast |w|^p)|w|^pdx \\
  & < &\frac{1}{2}t^2\|w\|^2 - \frac{t^{2p}}{2p}\int_{\mathbb{R}^N}(I_{\alpha}\ast |w|^p)|w|^pdx\\
  &=& (\frac{t^2}{2} - \frac{t^{2p}}{2p})\|w\|^2\\
  &< & (\frac{1}{2} - \frac{1}{2p})\|w\|^2 = m_\infty,
\end{eqnarray*}
Hence, $m < m_\infty$ and the proof is completed.
\end{proof}

\bibliographystyle{plain}

\begin{thebibliography}{}

\end{thebibliography}


\begin{thebibliography}{10}

\bibitem{AmbrosettiCeramiRuizSLCSSNAE2008}
Antonio Ambrosetti, Giovanna Cerami, and David Ruiz.
\newblock Solitons of linearly coupled systems of semilinear non-autonomous
  equations on {$\Bbb R^n$}.
\newblock {\em J. Funct. Anal.}, 254(11):2816--2845, 2008.

\bibitem{AveniaSicilianoSquassinaMMMAS2015}
Pietro d'Avenia, Gaetano Siciliano and Marco Squassina.
\newblock On fractional Choquard equations.
\newblock {\em Mathematical Models and Methods in Applied Sciences}, 25(8):1447--1476, 2015.


\bibitem{BerestyckiLionsARMA1983}
H.~Berestycki and P.-L. Lions.
\newblock Nonlinear scalar field equations. {I}. {E}xistence of a ground state.
\newblock {\em Arch. Rational Mech. Anal.}, 82(4):313--345, 1983.

\bibitem{CeramiVairaJDE2010}
Giovanna Cerami and Giusi Vaira.
\newblock Positive solutions for some non-autonomous {S}chr\"odinger-{P}oisson
  systems.
\newblock {\em J. Differential Equations}, 248(3):521--543, 2010.

\bibitem{ChenWenxiongLicongmingOubiaoCPAM2006}
Wenxiong Chen, Congming Li, and Biao Ou.
\newblock Classification of solutions for an integral equation.
\newblock {\em Comm. Pure Appl. Math.}, 59(3):330--343, 2006.


\bibitem{CingolaniClappSecchiZAMP}
Silvia Cingolani, M{\'o}nica Clapp, and Simone Secchi.
\newblock Multiple solutions to a magnetic nonlinear {C}hoquard equation.
\newblock {\em Z. Angew. Math. Phys.}, 63(2):233--248, 2012.

\bibitem{ClappSalazar2013}
M{\'o}nica Clapp and Dora Salazar.
\newblock Positive and sign changing solutions to a nonlinear {C}hoquard
  equation.
\newblock {\em J. Math. Anal. Appl.}, 407(1):1--15, 2013.

\bibitem{DiNezzaPalatucciValdinoci2012}
Eleonora Di~Nezza, Giampiero Palatucci, and Enrico Valdinoci.
\newblock Hitchhiker's guide to the fractional {S}obolev spaces.
\newblock {\em Bull. Sci. Math.}, 136(5):521--573, 2012.


\bibitem{GilbargTrudingerEPDE2001}
David Gilbarg and Neil~S. Trudinger.
\newblock {\em Elliptic partial differential equations of second order}.
\newblock Classics in Mathematics. Springer-Verlag, Berlin, 2001.
\newblock Reprint of the 1998 edition.

\bibitem{LiebElliot}
Elliott~H. Lieb.
\newblock Existence and uniqueness of the minimizing solution of {C}hoquard's
  nonlinear equation.
\newblock {\em Studies in Appl. Math.}, 57(2):93--105, 1976/77.

\bibitem{LiebLoss2001}
Elliott~H. Lieb and Michael Loss.
\newblock {\em Analysis}, volume~14 of {\em Graduate Studies in Mathematics}.
\newblock American Mathematical Society, Providence, RI, second edition, 2001.

\bibitem{LionsPLCERQNA1980}
P.-L. Lions.
\newblock The {C}hoquard equation and related questions.
\newblock {\em Nonlinear Anal.}, 4(6):1063--1072, 1980.

\bibitem{MaLiCPSSNCEARMA2010}
Li~Ma and Lin Zhao.
\newblock Classification of positive solitary solutions of the nonlinear
  {C}hoquard equation.
\newblock {\em Arch. Ration. Mech. Anal.}, 195(2):455--467, 2010.

\bibitem{MorozPenrose}
Irene~M. Moroz, Roger Penrose, and Paul Tod.
\newblock Spherically-symmetric solutions of the {S}chr\"odinger-{N}ewton
  equations.
\newblock {\em Classical Quantum Gravity}, 15(9):2733--2742, 1998.
\newblock Topology of the Universe Conference (Cleveland, OH, 1997).

\bibitem{MorozSchaftingenpreprint}
Vitaly Moroz and Jean Van~Schaftingen.
\newblock Existence of groundstates for a class of nonlinear {C}hoquard
  equations.
\newblock {\em Transactions of the American Mathematical Society, to appear}.

\bibitem{MorozVanSchaftingenCVPDE2015}
Vitaly Moroz and Jean Van~Schaftingen.
\newblock Semi-classical states for the {C}hoquard equation.
\newblock {\em Calc. Var. Partial Differential Equations}, 52(1-2):199--235,
  2015.

\bibitem{MorozSchaftingen2013}
Vitaly Moroz and Jean Van~Schaftingen.
\newblock Groundstates of nonlinear {C}hoquard equations: existence,
  qualitative properties and decay asymptotics.
\newblock {\em J. Funct. Anal.}, 265(2):153--184, 2013.

\bibitem{MorozSchaftingen2013JDE}
Vitaly Moroz and Jean Van~Schaftingen.
\newblock Nonexistence and optimal decay of supersolutions to {C}hoquard
  equations in exterior domains.
\newblock {\em J. Differential Equations}, 254(8):3089--3145, 2013.

\bibitem{Pekar1954}
S.~Pekar.
\newblock Untersuchung ¨¹ber die elektronentheorie der kristalle.
\newblock {\em Akademie Verlag, Berlin}, 1954.

\bibitem{MRieszActaMath}
Marcel Riesz.
\newblock L'int\'egrale de {R}iemann-{L}iouville et le probl\`eme de {C}auchy.
\newblock {\em Acta Math.}, 81:1--223, 1949.

\bibitem{WeiWinter2009}
Juncheng Wei and Matthias Winter.
\newblock Strongly interacting bumps for the {S}chr\"odinger-{N}ewton
  equations.
\newblock {\em J. Math. Phys.}, 50(1):012905, 22, 2009.

\bibitem{Wi:MT}
Michel Willem.
\newblock {\em Minimax theorems}.
\newblock Progress in Nonlinear Differential Equations and their Applications,
  24. Birkh\"auser Boston Inc., Boston, MA, 1996.

\bibitem{ZhangZhengjieTassiloHuailianXiaHongqiang}
Zhengjie Zhang, K{\"u}pper Tassilo, Ailian Hu, and Hongqiang Xia.
\newblock Existence of a nontrivial solution for {C}hoquard's equation.
\newblock {\em Acta Math. Sci. Ser. B Engl. Ed.}, 26(3):460--468, 2006.

\bibitem{zouwenmingschechter}
Wenming Zou and Martin Schechter.
\newblock {\em Critical point theory and its applications}.
\newblock Springer, New York, 2006.

\end{thebibliography}
\end{document}